\documentclass[11pt]{article}
\usepackage{amsmath,amsthm,amsfonts,amssymb,bm,wasysym}
\usepackage{epsfig}
\usepackage[usenames]{color}
\usepackage{verbatim}
\usepackage{hyperref}
\usepackage{multicol}
\usepackage{comment}
\usepackage{float}
\usepackage{graphicx}
\usepackage[utf8]{inputenc}

\parskip \medskipamount
\newtheorem{theorem}{Theorem}[section]

\numberwithin{equation}{section}
\newtheorem{lemma}[theorem]{Lemma}
\newtheorem{proposition}[theorem]{Proposition}
\newtheorem{corollary}[theorem]{Corollary}

\newtheorem{claim}[theorem]{Claim}

\numberwithin{equation}{section}

\def\N{\mathbb{N}}
\def\Z{\mathbb{Z}}

\def\EE{\mathcal{E}}

\def\NN{\mathcal{N}}

\renewcommand{\phi}{\varphi}
\renewcommand{\epsilon}{\varepsilon}

\def\RR{\mathcal{R}}
\allowdisplaybreaks

\newcommand{\1}{{\text{\Large $\mathfrak 1$}}}

\renewcommand{\emptyset}{\varnothing}

\newcommand{\til}{\widetilde}
\def\reff#1{(\ref{#1})}

\newcommand{\pr}[1]{\mathbb{P}\!\left(#1\right)}
\newcommand{\E}[1]{\mathbb{E}\!\left[#1\right]}
\newcommand{\estart}[2]{\mathbb{E}_{#2}\!\left[#1\right]}
\newcommand{\prstart}[2]{\mathbb{P}_{#2}\!\left(#1\right)}

\newcommand{\norm}[1]{\left\| #1 \right\|}

\newcommand{\tn}{|\kern-.1em|\kern-0.1em|}

\newcommand{\vr}[1]{\mathrm{Var}\left(#1\right)}

\newcommand{\cc}[1]{\mathrm{Cap}\left(#1\right)}

\newcommand{\red}[1]{{\color{red}{#1}}}

\newcommand\be{\begin{equation}}
\newcommand\ee{\end{equation}}

\newcommand{\td}[1]{\textbf{\red{[#1]}}}

\begin{document}
\title{\bf Capacity of the range of random walk on $\Z^d$}

\author{Amine Asselah \thanks{Aix-Marseille Universit\'e \& 
Universit\'e Paris-Est Cr\'eteil; amine.asselah@u-pec.fr} \and
Bruno Schapira\thanks{Aix-Marseille Universit\'e, CNRS, Centrale Marseille, I2M, UMR 7373, 13453 Marseille, France;  bruno.schapira@univ-amu.fr} \and Perla Sousi\thanks{University of Cambridge, Cambridge, UK;   p.sousi@statslab.cam.ac.uk} 
}
\date{}
\maketitle
\begin{abstract}
We study the capacity of the range of a transient
simple random walk on $\Z^d$.
Our main result is a
central limit theorem for the capacity of the range for~$d\ge 6$. 
We present a few open questions in lower dimensions. 
\newline
\newline
\emph{Keywords and phrases.} Capacity,  Green kernel, Lindeberg-Feller central limit theorem.
\newline
MSC 2010 \emph{subject classifications.} Primary 60F05, 60G50.
\end{abstract}

\section{Introduction}\label{sec:intro}
This paper is devoted to the study of the capacity of the range of
a transient random walk on $\Z^d$. 
Let $\{S_k\}_{k\ge 0}$ be a simple random walk in dimension $d\geq 3$.
For any integers $m$ and $n$,
we define the range $\RR[m,n]$ to be the set of visited sites
during the interval $[m,n]$, i.e.
\[
\RR[m,n]= \{ S_m,\ldots, S_n\}.
\]
We write simply $\RR_n= \RR[0,n]$. 
We recall that the capacity of a finite set $A\subseteq \Z^d$ 
is defined to be
\[
\cc{A} = \sum_{x\in A} \prstart{T_A^+=\infty}{x},
\]
where $T_A^+=\inf\{t\geq 1: S_t\in A\}$ is the first return time to $A$. 

The capacity of the range of a walk has a long history.
Jain and Orey~\cite{JainOrey} proved, some fifty years ago,
 that $\cc{\RR_n}$ satisfies a law of large numbers for all $d\geq 3$, i.e.\ almost surely
\[
 \lim_{n\to\infty}
\frac{\cc{\RR_n}}{n} = \alpha_d.
\]
Moreover, they showed that $\alpha_d>0$ if and only if $d\geq 5$. 
In the eighties, Lawler established estimates on intersection
probabilities for random walks, which are relevant tools for estimating
the expected capacity of the range (see \cite{Lawlerinter}).
Recently, the study of random interlacements by Sznitman \cite{S10},
has given some momentum to the study of the capacity of the union
of the ranges of a collection of independent walks.
In order to obtain bounds on the capacity of such union of ranges,
R\'ath and Sapozhnikov in~\cite{RathSap} have obtained 
bounds on the capacity of the range of a simple transient walk. The capacity of the range is a natural object to probe the geometry of the walk under
localisation constraints. For instance, the first two authors have used the capacity
of the range in~\cite{AS2} to characterise the walk conditioned on
having a small range.


In the present paper, we establish a central limit theorem for $\cc{\RR_n}$ when $d\geq 6$.

\begin{theorem}\label{thm:clt}
For all $d\geq 6$, there is a positive constant $\sigma_d$ such that
\[
\frac{\cc{\RR_n} - \E{\cc{\RR_n}}}{\sqrt{n}} \Longrightarrow\sigma_d
 \NN(0,1),\quad \text{as } n\to \infty,
\]
where $\Longrightarrow$ denotes convergence in distribution,
and $\NN(0,1)$ denotes a standard normal random variable.
\end{theorem}

A key tool in the proof of Theorem~\ref{thm:clt} is the following inequality.

\begin{proposition}\label{prop:capdec}
Let $A$ and $B$ be finite subsets of $\Z^d$. Then, 
\be\label{main-lower}
\cc{A\cup B}\ge
\cc{A} + \cc{B} - 2\sum_{x\in A} \sum_{y\in B}G(x,y),
\ee
where $G$ is Green's kernel for a simple random walk in $\Z^d$
\[
G(x,y) = \estart{\sum_{t=0}^{\infty}\1(X_t=y)}{x}.
\]
\end{proposition}
Note in comparison the well known upper bound (see for instance~\cite[Proposition 2.2.1]{Lawlerinter})
\be\label{key-lawler}
\cc{A\cup B}\le \cc{A}+\cc{B}-\cc{A\cap B}
\ee

In dimension four, asymptotics of $\E{\cc{\RR_n}}$
can be obtained from Lawler's estimates on 
non-intersection probabilities for three 
random walks, that we recall here for convenience.
\begin{theorem}{\rm (\cite[Corollary 4.2.5]{Lawlerinter})}\label{thm:lawler}
Let $\RR^1,\RR^2$ and $\RR^3$ be the ranges of
three independent random walks in $\Z^4$ starting at 0. Then,
\be\label{lawler-key}
\lim_{n\to\infty} \ \log n\times 
\pr{ \RR^1[1,n]\cap( \RR^2[0,n]\cup \RR^3[0,n])=\emptyset,\ 
0\not\in \RR^3[1,n]}=\frac{\pi^2}{8},
\ee
and 
\be\label{lawler-key2}
\lim_{n\to\infty} \ \log n\times 
\pr{\RR^1[1,\infty)\cap( \RR^2[0,n]\cup \RR^3[0,n])=\emptyset,\ 
0\not\in \RR^3[1,n]}=\frac{\pi^2}{8}.
\ee
\end{theorem}

Actually \eqref{lawler-key2} is not stated exactly in this form in~\cite{Lawlerinter}, but it can be proved using exactly the same proof as for equation (4.11) in~\cite{Lawlerinter}. 
As mentioned, we deduce from this result, the following estimate for the mean of the capacity. 
\begin{corollary}\label{lem:d4}
Assume that $d=4$. Then, 
\be\label{bounds-d4}
 \lim_{n\to\infty}\ \frac{\log n}{n}\ \E{\cc{\RR_n}}\ = \ \frac{\pi^2}{8}.
\ee
\end{corollary}

In dimension three, we use the following representation of capacity (see \cite[Lemma~2.3]{JainOrey-properties})
\be\label{variation-capa}
\cc{A} = \frac{1}{\inf_{\nu} \sum_{x\in A}\sum_{y\in A} G(x,y)\nu(x)\nu(y)},
\ee
where the infimum is taken over all probability measures~$\nu$
supported on $A$. We obtain the following bounds:
\begin{proposition}\label{lem:d3}
Assume that $d=3$. There are positive constants $c$ and
 $C$, such that
\be\label{bounds-d3}
c \sqrt{n}\ \le\  \E{\cc{\RR_n}}\ \le\  C\sqrt{n}.
\ee
\end{proposition}
The rest of the paper is organised as follows. 
In Section~\ref{sec-two} we present the decomposition of the range,
which is at the heart of our central limit theorem. The capacity
of the range is cut into a {\it self-similar} part and an {\it error
term} that we bound in Section~\ref{sec-three}. In Section~\ref{sec-four}
we check Lindeberg-Feller's conditions. We deal with dimension three
and four in Section~\ref{sec-five}. Finally, we present some open
questions in Section~\ref{sec-six}.

\textbf{Notation:}
When $0\leq a\leq b$ are real numbers, we write $\RR[a,b]$ to denote $\RR[[a],[b]]$, where $[x]$ stands for the integer part of $x$. 
We also write $\RR_a$ for $\RR[0,[a]]$, and $S_{n/2}$ for $S_{[n/2]}$. 

For positive functions $f,g$ we write $f(n) \lesssim g(n)$ if there exists a constant $c > 0$ such that $f(n) \leq c g(n)$ for all $n$.  We write $f(n) \gtrsim g(n)$ if $g(n) \lesssim f(n)$.  Finally, we write $f(n) \asymp g(n)$ if both $f(n) \lesssim g(n)$ and $f(n) \gtrsim g(n)$.

\section{Decomposition for capacities}\label{sec-two}

\begin{proof}[\bf Proof of Proposition~\ref{prop:capdec}]
Note first that by definition,
\begin{align*}
\cc{A\cup B}&=\cc{A} + \cc{B} - \sum_{x\in A\setminus B} \prstart{T_A^+=\infty, T_B^+<\infty}{x} 
\\&- \sum_{x\in A\cap B} \prstart{T_A^+=\infty, T_B^+<\infty}{x} -\sum_{x\in B\setminus A} \prstart{T_A^+<\infty, T_B^+=\infty}{x}  \\&- \sum_{x\in A\cap B} \prstart{T_A^+<\infty, T_B^+=\infty}{x} - \sum_{x\in A\cap B} \prstart{T_A^+=\infty, T_B^+=\infty}{x}\\
&\geq \cc{A} + \cc{B} - \sum_{x\in A\setminus B} \prstart{T_B^+<\infty}{x} - \sum_{x\in B\setminus A} \prstart{T_A^+<\infty}{x}
- |A\cap B|.
\end{align*}
For any finite set $K$ and all $x\notin K$ by considering the last visit to $K$ we get
\begin{align*}
\prstart{T_K^+<\infty}{x} =\sum_{y\in K} G(x,y) \prstart{T_K^+=\infty}{y}.
\end{align*}
This way we obtain
\begin{align*}
\sum_{x\in A\setminus B} \prstart{T_B^+<\infty}{x} \leq \sum_{x\in A\setminus B} \sum_{y\in B} G(x,y) \quad \text{and} \quad \sum_{x\in B\setminus A} \prstart{T_A^+<\infty}{x} \leq \sum_{x\in B\setminus A} \sum_{y\in A} G(x,y).
\end{align*}
Hence we get
\begin{align*}
\cc{A\cup B} \geq \cc{A} + \cc{B} - 2\sum_{x\in A} \sum_{y\in B} G(x,y) + \sum_{x\in A\cap B} \sum_{y\in A} G(x,y) \\+ \sum_{x\in A\cap B} \sum_{y\in B} G(x,y) - |A\cap B|.
\end{align*}
Since $G(x,x)\geq 1$ for all $x$ we get 
\[
\sum_{x\in A\cap B} \sum_{y\in A} G(x,y) \geq |A\cap B|
\]
and this concludes the proof of the 
lower bound and also the proof of the lemma.
\end{proof}

The decomposition of $\cc{\RR_n}$ 
stated in the following corollary is crucial in the rest of the paper.
\begin{corollary}\label{cor:decomposition}
For all $L$ and $n$, with $2^L\le n$, we have 
\begin{align*}
 \sum_{i=1}^{2^L} \cc{\RR^{(i)}_{n/2^L}} - 2\sum_{\ell=1}^{L}\sum_{i=1}^{2^{\ell-1}} \EE_{\ell}^{(i)}\leq \cc{\RR_n}\leq \sum_{i=1}^{2^L} \cc{\RR^{(i)}_{n/2^L}},
\end{align*}
where $(\cc{\RR^{(i)}_{n/2^L}},\ i=1,\dots,2^L)$ are independent 
and $\RR^{(i)}_{n/2^L}$ has the same law as $\RR_{[n/2^L]}$ or~$\RR_{[n/2^L+1]}$ and for each $\ell$ the random variables $(\EE_{\ell}^{(i)})_i$ are independent and have the same law as $\sum_{x\in \RR^{(i)}_{n/2^L}} \sum_{y\in \til{\RR}^{(i)}_{n/2^L}} G(x,y)$, with $\til{\RR}$ an independent copy of $\RR$.
\end{corollary}

\begin{proof}[\bf Proof]


Since we work on~$\Z^d$, the capacity is translation invariant, i.e.\ $\cc{A} = \cc{A+x}$ for all $x$, and hence it follows that 
\[
\cc{\RR_n} = \cc{\left(\RR_{n/2}-S_{n/2}\right) \cup \left(\RR[n/2,n]-S_{n/2}\right)}.
\]
The advantage of doing this is that now by the Markov property the random variables $\RR_{n/2}^{(1)}=\cc{\RR_{n/2}-S_{n/2}}$ and $\RR_{n/2}^{(2)}=\cc{\RR[n/2,n] -S_{n/2}}$ are independent. Moreover, by reversibility, each of them has the same law as the range of a simple random walk started from $0$ and run up to time $n/2$. Applying Proposition~\ref{prop:capdec} we get
\begin{align}\label{eq:key}
\cc{\RR_n}\geq \cc{\RR^{(1)}_{n/2}} + \cc{\RR^{(2)}_{n/2}} - 2\sum_{x\in \RR^{(1)}_{n/2}}\sum_{y\in \RR^{(2)}_{n/2}} G(x,y).
\end{align}
Applying the same subdivision to each of the terms $\RR^{(1)}$ and~$\RR^{(2)}$ and iterating $L$ times, we obtain
\begin{align*}
\cc{\RR_n} \geq \sum_{i=1}^{2^L} \cc{\RR^{(i)}_{n/2^L}} - 2\sum_{\ell=1}^{L}\sum_{i=1}^{2^{\ell-1}} \EE_{\ell}^{(i)}.
\end{align*}
Here $\EE_\ell^{(i)}$ has the same 
law as $\sum_{x\in \RR_{n/2^l} }\sum_{y\in \RR'_{n/2^l}}G(x,y)$, with $\RR'$ independent of $\RR$ and the random variables 
$(\EE_\ell^{(i)},\ i=1,\dots,2^l)$ are independent. Moreover,
 the random variables $(\RR_{n/2^L}^{(i)},\ i=1,\dots,2^L)$ are independent.
Using~\eqref{key-lawler} for the upper bound on $\cc{\RR_n}$ we get overall
\begin{align*}
\sum_{i=1}^{2^L} \cc{\RR^{(i)}_{n/2^L}} - 2\sum_{\ell=1}^{L}\sum_{i=1}^{2^{\ell-1}} \EE_{\ell}^{(i)}\leq \cc{\RR_n} \leq \sum_{i=1}^{2^L} \cc{\RR^{(i)}_{n/2^L}}
\end{align*}
and this concludes the proof.
\end{proof}

\section{Variance of $\cc{\RR_n}$ and error term}\label{sec-three}
As outlined in the Introduction, we want to apply the Lindeberg-Feller theorem to obtain the central limit theorem. In order to do so, we need to control the 
{\it error term} appearing in the decomposition of $\cc{\RR_n}$ in Corollary~\ref{cor:decomposition}. Moreover, we need to show that the variance of $\cc{\RR_n}/n$ converges to a strictly positive constant as $n$
tends to infinity. This is the goal of this section.
\subsection{On the {\it error term}}

We write $G_n(x,y)$ for the Green kernel up to time $n$, i.e.\ ,
\[
G_n(x,y) = \estart{\sum_{k=0}^{n-1}\1(S_k=y)}{x}.
\]
We now recall a well-known bound 
(see for instance~\cite[Theorem~4.3.1]{LawlerLimic})
\begin{align}\label{eq:wellknownbound}
	G(0,x) \leq \frac{C}{1+\norm{x}^{d-2}},
\end{align}
where $C$ is a positive constant. We start with a preliminary result.
\begin{lemma}\label{lem:rearrange}
For all $a\in \Z^d$ we have 
\begin{align*}\label{eq:rearrange}
\sum_{x\in \Z^d}\sum_{y\in \Z^d} G_n(0,x) G_n(0,y) G(0,x-y-a) \leq \sum_{x\in \Z^d}\sum_{y\in \Z^d} G_n(0,x) G_n(0,y) G(0,x-y). \end{align*}
Moreover, 
\begin{align*}
	\sum_{x\in \Z^d}\sum_{y\in \Z^d} G_n(0,x) G_n(0,y) G(0,x-y)\lesssim f_d(n),
\end{align*}
where 
\be\label{def-f}
f_5(n) = \sqrt{n}, \qquad 
f_6(n) = \log n,\quad\text{ and}\quad f_d(n) = 1 \quad\forall d\geq 7.
\ee
\end{lemma}

\begin{proof}[\bf Proof]

Let $S_a = \sum_{x,y} G_n(0,x) p_{2k}(x,y+a) G_n(0,y)$. Since
\[
p_{2k}(x,y-a) = \sum_z p_{k}(x,z) p_k(z,y-a) = \sum_{z} p_k(z,x) p_k(z,y-a)
\]
letting $F_a(z) = \sum_y G_n(0,y) p_k(z,y+a)$ we have 
\begin{align}\label{eq:zya}
F_a(z) = \sum_y G_n(0,y) p_k(z-a,y),\quad\text{and}\quad
S_a = \sum_z F_0(z) F_a(z).
\end{align}
By Cauchy-Schwartz, we obtain
\[
S_a^2\leq \sum_z F_0^2(z) \cdot \sum_z F_a^2(z).
\]
Notice however that a change of variable and using~\eqref{eq:zya} yields
$$\sum_z F_a^2(z) = \sum_w F_a^2(w-a) = \sum_w F_0^2(w),$$ and hence we deduce
\[
S_a^2\leq S_0^2 \quad \forall \, a.
\]
We now note that if $X$ is a lazy simple random walk, then the sums in the statement of the lemma will only be affected by a multiplicative constant. So it suffices to prove the result for a lazy walk. It is a standard fact (see for instance~\cite[Proposition~10.18]{LevPerWil}) that the transition matrix of a lazy chain can be written as the square of another transition matrix. This now concludes the proof of the first inequality.

To simplify notation we write $G_n(x)=G_n(0,x)$ and $G(x)= G(0,x)$.
To prove the second inequality we split the second sum appearing in the statement of the lemma into three parts as follows
\begin{align}\label{eq:sums}
\begin{split}
\sum_{x}\sum_{y} &G_n(x) G_n(y) G(x-y) \leq  \sum_{\substack{\norm{x}\leq \sqrt{n}\\ \norm{y}\leq \sqrt{n}}}G_n(x) G_n(y) G(x-y) \\+ &2\sum_{\substack{\norm{x}\geq \sqrt{n}\\ \frac{\sqrt{n}}{2}\leq \norm{y}\leq \sqrt{n}}}G_n(x)G_n(y)G(x-y) +2 \sum_{\substack{\norm{x}\geq \sqrt{n}\\ \norm{y}\leq \frac{\sqrt{n}}{2}}}G_n(x)G_n(y)G(x-y)\\ &\qquad \qquad \qquad\quad \qquad =:I_1 + I_2 + I_3,
\end{split}
\end{align}
where $I_k$ is the $k$-th sum appearing on the right hand side of the inequality above. 
The first sum~$I_1$ is upper bounded by
\begin{align*}
2\sum_{k=0}^{\frac{\log_2(n)}{2}}\sum_{\frac{\sqrt{n}}{2^{k+1}}\leq \norm{x}\leq \frac{\sqrt{n}}{2^k}}\left( \sum_{\norm{y}\leq \frac{\sqrt{n}}{2^{k+2}}}  G_n(x)G_n(y)G(x-y) + \sum_{r=0}^{\frac{\sqrt{n}}{2^k}}\sum_{\substack{y:\,\norm{y-x}=r\\ \norm{x}\geq  \norm{y}\geq \frac{\sqrt{n}}{2^{k+2}}}} G_n(x) G_n(y)G(x-y) \right).
\end{align*}
For any fixed $k\leq \log_2(n)/2$, using~\eqref{eq:wellknownbound} we get
\begin{align}\label{eq:i1first}
\begin{split}
	\sum_{\frac{\sqrt{n}}{2^{k+1}}\leq \norm{x}\leq \frac{\sqrt{n}}{2^k}} \sum_{\norm{y}\leq \frac{\sqrt{n}}{2^{k+2}}}  G_n(x)G_n(y)G(x-y) &\lesssim
	\left(\frac{\sqrt{n}}{2^{k}} \right)^d \left(\frac{\sqrt{n}}{2^k}\right)^{4-2d}\sum_{\norm{y}\leq \frac{\sqrt{n}}{2^{k+2}}}G_n(y) \\
	& \lesssim \left(\frac{\sqrt{n}}{2^{k}} \right)^{4-d}\cdot  \sum_{r=1}^{\frac{\sqrt{n}}{2^{k+2}}} \frac{r^{d-1}}{r^{d-2}} \asymp \left(\frac{\sqrt{n}}{2^{k}} \right)^{6-d}.
\end{split}
\end{align}
Similarly using~\eqref{eq:wellknownbound} again for any fixed $k\leq \log_2(n)/2$ we can bound
\begin{align}\label{eq:i1second}
	\sum_{\frac{\sqrt{n}}{2^{k+1}}\leq \norm{x}\leq \frac{\sqrt{n}}{2^k}}\sum_{r=0}^{\frac{\sqrt{n}}{2^k}}\sum_{\substack{y:\,\norm{y-x}=r\\ \norm{x}\geq  \norm{y}\geq \frac{\sqrt{n}}{2^{k+2}}}} G_n(x) G_n(y)G(x-y) \lesssim \left(\frac{\sqrt{n}}{2^{k}} \right)^{4-d} \sum_{r=1}^{\frac{\sqrt{n}}{2^k}} \frac{r^{d-1}}{r^{d-2}}\asymp \left(\frac{\sqrt{n}}{2^{k}} \right)^{6-d}.
\end{align}
Therefore using~\eqref{eq:i1first} and~\eqref{eq:i1second} and summing over all $k$ yields
\begin{align*}
	I_1\lesssim f_d(n).
\end{align*}
We now turn to bound $I_2$. From~\eqref{eq:wellknownbound} we have 
\begin{align*}
	I_2 &\lesssim \sum_{\substack{\norm{x}\geq 2\sqrt{n} \\ \frac{\sqrt{n}}{2}\leq  \norm{y}\leq \sqrt{n}} } G_n(x) G_n(y) G(x-y) + \sum_{\substack{\sqrt{n}\leq \norm{x}\leq 2\sqrt{n} \\ \frac{\sqrt{n}}{2}\leq  \norm{y}\leq \sqrt{n}} } G_n(x) G_n(y) G(x-y) \\
	&\lesssim n^2 \cdot \frac{1}{(\sqrt{n})^{d-2}} + (\sqrt{n})^{4-d} \sum_{r=1}^{\sqrt{n}} \frac{r^{d-1}}{r^{d-2}} \asymp f_d(n),\end{align*}
where for the first sum we used that 
$\sum_x G_n(x) =n$. Finally, $I_3$ is treated similarly as above to
yield
\begin{align*}
	I_3\lesssim n^2 \cdot\frac{1}{(\sqrt{n})^{d-2}} \asymp f_d(n).
\end{align*}
Putting all these bounds together concludes the proof.
\end{proof}

\begin{lemma}\label{lem:powers}
For all $n$, let $\RR_n$ and $\til{\RR}_n$ be the ranges up to time $n$ of two independent simple random walks in $\Z^d$ started from~$0$.
For all $k,n\in \N$ we have 
\[
\E{\left( \sum_{x\in \RR_n}\sum_{y\in \til{\RR}_n} G(x,y)\right)^k} \leq C(k)(f_d(n))^k,
\]
where $f_d(n)$ is the function defined in the statement of Lemma~\ref{lem:rearrange} and $C(k)$ is a constant that depends only on $k$.
\end{lemma}

\begin{proof}[\bf Proof]

Let $L_\ell(x)$ denote the local time at $x$ up to time $\ell$ for the random walk $S$, i.e.\
\[
L_\ell(x) = \sum_{i=0}^{\ell-1}\1(S_i=x).
\]
Let $\til{S}$ be an independent walk and $\til{L}$ denote its local times. 
Then, we get
\begin{align*}
\sum_{x\in \RR_n} \sum_{y\in \til{\RR}_n} G(x,y) 
\le \sum_{x\in \Z^d}\sum_{y\in \Z^d} L_n(x) \til{L}_n(y) G(x,y).
\end{align*}
So, for $k=1$ by independence, we get using Lemma~\ref{lem:rearrange}
\begin{align*}
\E{\sum_{x\in \RR_n} \sum_{y\in \til{\RR}_n} G(x,y)} \le 
\sum_{x\in \Z^d} \sum_{y\in \Z^d} G_n(0,x) G_n(0,y) G(0,x-y) \lesssim f_d(n),
\end{align*}
As in Lemma~\ref{lem:rearrange} to simplify notation we write $G_n(x) = G_n(0,x)$.

For the $k$-th moment we have
\begin{align}\label{kthmoment}
\E{\left( \sum_{x\in \RR_n} \sum_{y\in \til{\RR}_n} G(x,y)\right)^k} 
\le 
 \sum_{x_1,\ldots, x_k} \sum_{y_1,\ldots, y_k} \E{\prod_{i=1}^{k}L_n(x_i)} \E{\prod_{i=1}^{k}L_n(y_i)} \prod_{i=1}^{k}G(x_i-y_i).
\end{align}
For any $k$-tuples $x_1,\ldots, x_k$ and $y_1,\ldots,y_k$, we have
\begin{align*}
	&\E{\prod_{i=1}^{k} L_n(x_i)} \leq \sum_{\sigma:\,\text{permutation of } \{1,\ldots,k\}}G_n(x_{\sigma(1)})\prod_{i=2}^{k}G_n(x_{\sigma(i)}- x_{\sigma(i-1)}) \quad \text{and}
\\ &\E{\prod_{i=1}^{k} L_n(y_i)} \leq \sum_{\pi:\,\text{permutation of } \{1,\ldots,k\}}G_n(y_{\pi(1)})\prod_{i=2}^{k}G_n(y_{\pi(i)}- y_{\pi(i-1)}).
\end{align*}
Without loss of generality, we consider the term corresponding to the identity permutation for $x$ and a permutation $\pi$ for $y$. 
Then, the right hand side of \reff{kthmoment} is a sum
of terms of the form
\begin{align*}
 G_n(x_1) G_n(x_2-x_1) \ldots G_n(x_k-x_{k-1})G_n(y_{\pi(1)})G_n(y_{\pi(2)}-y_{\pi(1)})\ldots G_n(y_{\pi(k)}-y_{\pi(k-1)})\prod_{i=1}^{k}G(x_i-y_i).
\end{align*}
Suppose now that the term $y_k$ appears in two terms 
in the above product, i.e.
\[
G_n(y_k - y_{\pi(i)}) G_n(y_k-y_{\pi(j)}).
\]
By the triangle inequality we have that one of the following two inequalities has to be true
\[
\|y_k - y_{\pi(i)}\| \geq \frac{1}{2} \| y_{\pi(i)} - y_{\pi(j)}\| 
\quad \text{or}\quad
\|y_k - y_{\pi(j)}\| \geq \frac{1}{2} \| y_{\pi(i)} - y_{\pi(j)}\|.
\]
Since Green's kernel is radially decreasing and satisfies $G(x)\asymp |x|^{2-d}$ for $\|x\|>1$ we get 
\[
G_n(y_k - y_{\pi(i)}) G_n(y_k-y_{\pi(j)}) \lesssim G_n(y_{\pi(j)} - y_{\pi(i)})\left(  G_n(y_k-y_{\pi(j)}) + G_n(y_k-y_{\pi(i)})\right).
\]
Plugging this upper bound into the product and summing only over $x_k$ and $y_k$, while fixing the other terms, we obtain
\begin{align*}
&\sum_{x_k, y_k} G_n(x_k-x_{k-1}) G_n(y_k-y_{\pi(i)}) G(x_k-y_k) \\
&= \sum_{x_k, y_k}  G_n(x_k-x_{k-1}) G_n(y_k-y_{\pi(i)}) G((x_k-x_{k-1})-(y_k-y_{\pi(i)}))\\&=
\sum_{x,y} G_n(x) G_n(y) G((x-y)-(x_{k-1}-y_{\pi(i)}))\lesssim f_d(n),
\end{align*}
where the last inequality follows from Lemma~\ref{lem:rearrange}. 
Continuing by induction completes the proof.
\end{proof}

\subsection{On the variance of $\cc{\RR_n}$}
\begin{lemma}\label{lem:variance}
For $d\geq 6$ there exists a strictly positive constant $\gamma_d$ so that 
\[
\lim_{n\to\infty} \frac{\vr{\cc{\RR_n}}}{n} =\gamma_d>0.
\]
\end{lemma}

We split the proof of the lemma above in two parts. First we establish the existence of the limit and then we show it is strictly positive. For the existence, we need to use Hammersley's lemma~\cite{Hammersley}, which we recall here.

\begin{lemma}[Hammersley]\label{lem:ham}
Let $(a_n), (b_n), (c_n)$ be three sequences of real numbers satisfying for all~$n,m$
\[
a_n+a_m - c_{n+m}\leq a_{n+m}\leq a_n+a_m + b_{n+m}.
\]
If the sequences $(b_n), (c_n)$ are positive and non-decreasing and additionally satisfy
\[
\sum_{n=1}^{\infty} \frac{b_n+c_n}{n(n+1)} <\infty,
\]
then the limit as $n\to \infty$ of $a_n/n$ exists.
\end{lemma}

For a random variable $X$ we will write $\overline{X}=X-\E{X}$.

\begin{lemma}\label{lem:exis}
For $d\geq 6$, the limit as $n$ tends to infinity 
of $\vr{\cc{\RR_n}}/n$ exists. 
\end{lemma}

\begin{proof}[\bf Proof]

We follow closely the proof of Lemma~6.2 of Le Gall~\cite{Le-Gall}. 
To simplify notation we write $X_n=\cc{\RR_n}$,
and we set for all $k\geq 1$
\[
a_k=\sup\left\{ \sqrt{\E{\overline{X}_n^2}}: \, 2^k\leq n< 2^{k+1}\right\}.
\]
For $k\geq 2$, take $n$ such that $2^k\leq n<2^{k+1}$ 
and write $\ell=[n/2]$ and $m=n-\ell$. 
Then, from Corollary~\ref{cor:decomposition} for $L=1$ we get 
\begin{align*}
X^{(1)}_{\ell} + X^{(2)}_{m} - 2 \EE_\ell\leq X_n \leq X^{(1)}_{\ell} + X^{(2)}_{m},
\end{align*}
where $X^{(1)}$ and $X^{(2)}$ are independent and $\EE_\ell$ has the same law as $\sum_{x\in \RR_\ell}\sum_{y\in\til{\RR}_m} G(x,y)$ with~$\til{\RR}$ an independent copy of $\RR$.

 Taking expectations and subtracting we obtain
\begin{align*}
|\overline{X}_n - (\overline{X}^{(1)}_{\ell} + \overline{X}^{(2)}_{m} )|\leq 2\max\left(\EE_\ell, \E{\EE_\ell}\right).
\end{align*}
Since $\overline{X}^{(1)}$ and $\overline{X}^{(2)}$ are independent, we get
\[
\norm{\overline{X}^{(1)}_{\ell} + \overline{X}^{(2)}_{m}}_2 = \left( \norm{\overline{X}^{(1)}_{\ell}}_2^2
+\norm{\overline{X}^{(2)}_{m}}_2^2\right)^{1/2}.
\] 
By the triangle inequality we now obtain
\begin{align*}
\|\overline{X}_n\|_2&\leq \| \overline{X}^{(1)}_{\ell}+ \overline{X}^{(2)}_{m}\|_2 + \|2\max(\EE_\ell,\E{\EE_\ell})\|_2 \\
&\leq \left( \norm{\overline{X}^{(1)}_{\ell}}_2^2
+\norm{\overline{X}^{(2)}_{m}}_2^2\right)^{1/2} + 2\left(\norm{\EE_\ell}_2 + \E{\EE_\ell} \right) \leq 
\left( \norm{\overline{X}^{(1)}_{\ell}}_2^2
+\norm{\overline{X}^{(2)}_{m}}_2^2\right)^{1/2} + c_1 f_d(n) \\ &\leq \left( \norm{\overline{X}^{(1)}_{\ell}}_2^2
+\norm{\overline{X}^{(2)}_{m}}_2^2\right)^{1/2} + c_1 \log n,
\end{align*}
where $c_1$ is a positive constant. The penultimate inequality follows from Lemma~\ref{lem:powers}, 
and for the last inequality we used that $f_d(n)\leq \log n$ 
for all $d\geq 6$. From the definition of $a_k$, we deduce that 
\begin{align*}
a_k\leq 2^{1/2} a_{k-1} +c_2 k,
\end{align*}
for another positive constant $c_2$.
Setting $b_k=a_k k^{-1}$ gives for all $k$ that
\begin{align*}
b_{k}\leq 2^{1/2} b_{k-1} + c_2,
\end{align*}
and hence $b_k\lesssim 2^{k/2}$, which implies that $a_k\lesssim k \cdot 2^{k/2}$ for all $k$. This gives that for all $n$
\begin{align}\label{eq:roughbound}
\vr{\overline{X}_n} \lesssim n(\log n)^2.
\end{align}
Proposition~\ref{prop:capdec} and~\eqref{key-lawler}  give that for all $n,m$ 
\begin{align*}
X^{(1)}_n + X_{m}^{(2)} - 2\EE(n,m)\leq X_{n+m}\leq X^{(1)}_n + X_{m}^{(2)},
\end{align*}
where again $X^{(1)}$ and $X^{(2)}$ are independent and 
\begin{align}\label{eq:useful}
\EE(n,m)=\sum_{x\in \RR_n}\sum_{y\in \til{\RR}_m} G(x,y) \leq\sum_{x\in \RR_{n+m}}\sum_{y\in \til{\RR}_{n+m}} G(x,y) 
\end{align}
with $\RR$ and $\til{\RR}$ independent. As above we get
\begin{align*}
\left|\overline{X}_{n+m} - \left( \overline{X}^{(1)}_n + \overline{X}_{m}^{(2)} \right)\right| \leq 2\max(\EE(n,m), \E{\EE(n,m)})
\end{align*}
and by the triangle inequality again 
\begin{align*}
\left|\norm{\overline{X}_{n+m}}_2 - \norm{\overline{X}^{(1)}_n + \overline{X}_{m}^{(2)} }_2 \right|  \leq 4\norm{\EE(n,m)}_2.
\end{align*}
Taking the square of the above inequality and using that $\overline{X}^{(1)}_n$ and $\overline{X}^{(2)}_m$ are independent we obtain
\begin{align*}
\norm{\overline{X}_{n+m}}_2^2 &\leq \norm{\overline{X}_n}_2^2 + \norm{\overline{X}_m}_2^2+ 8\sqrt{\norm{\overline{X}_n}_2^2 +\norm{\overline{X}_m}_2^2} \norm{\EE(n,m)}_2 + 16\norm{\EE(n,m)}_2^2\\
\norm{\overline{X}_n}_2^2 + \norm{\overline{X}_m}_2^2 &\leq \norm{\overline{X}_{n+m}}_2^2 + 8\norm{\overline{X}_{n+m}}_2 \norm{\EE(n,m)}_2+
16\norm{\EE(n,m)}_2^2.
\end{align*}
We set $\gamma_n=\norm{\overline{X}_n}_2^2$, $d_n=c_1 \sqrt{n}
(\log n)^2$ and $d_n'=c_2\sqrt{n}(\log n)^2$, where $c_1$ and $c_2$ are two positive constants. 
Using the bound from~\eqref{eq:roughbound} together with~\eqref{eq:useful} and Lemma~\ref{lem:powers} in the inequalities above  yields 
\begin{align*}
\gamma_n + \gamma_m - d'_{n+m}\leq  \gamma_{n+m} \leq \gamma_n +\gamma_m + d_{n+m}.
\end{align*}
We can now apply Hammersley's result, Lemma~\ref{lem:ham}, to deduce that the limit $\gamma_n/n$ exists, i.e. 
\[
\lim_{n\to \infty}\frac{\vr{\overline{X}_n}}{n} = \gamma_d\geq 0
\]
and this finishes the proof on the existence of the limit.
\end{proof}

\subsection{Non-degeneracy: $\gamma_d>0$}\label{sec:gammapos}

To complete the proof of Lemma~\ref{lem:variance} we need to show that the limit $\gamma$ is strictly positive. We will achieve this by using the same trick of not allowing double-backtracks at even times (defined below) as in~\cite[Section~4]{AS1}.

As in~\cite{AS1} we consider a walk with no double backtracks at even times. A walk makes a double backtrack at time~$n$ if $S_{n-1}=S_{n-3}$ and $S_n=S_{n-2}$. Let $\til{S}$ be a walk with no double backtracks at even times constructed as follows: we set $\til{S}_0=0$ and let $\til{S}_1$ be a random neighbour of $0$ and $\til{S}_2$ a random neighbour of $\til{S}_1$. Suppose we have constructed $\til{S}$ for all times $k\leq 2n$, then we let $(\til{S}_{2n+1}, \til{S}_{2n+2})$ be uniform in the set
\[
\{(x,y): \quad \|x-y\| = \|\til{S}_{2n}-x\| =1 \,\text{  and  }\, (x,y)\neq (\til{S}_{2n-1}, \til{S}_{2n})\}.
\]
Having constructed $\til{S}$ we can construct a simple random walk in $\Z^d$ by adding a geometric number of double backtracks to $\til{S}$ at even times. More formally, let $(\xi_i)_{i=2,4,\ldots}$ be i.i.d.\ geometric random variables with mean $p/(1-p)$ and 
\[
\pr{\xi=k} = (1-p) p^k \quad \forall\, k\geq 0,
\]
where $p=1/(2d)^2$. Setting 
\[
N_k = \sum_{\substack{i=2 \\ i\text{ even}}}^{k} \xi_i,
\]
we construct $S$ from $\til{S}$ as follows. First we set $S_i=\til{S}_i$ for all $i\leq 2$ and for all $k\geq 1$ we set $I_k=[2k+ 2N_{2(k-1)} +1, 2k+2N_{2k}]$. If $I_k\neq \emptyset$, then if $i\in I_k$ is odd,  we set $S_i = \til{S}_{2k-1}$, while if $i$ is even, we set $S_i = \til{S}_{2k}$. Afterwards, for the next two time steps, we follow the path of $\til{S}$, i.e., 
\[
S_{2k+2N_{2k}+1} = \til{S}_{2k+1} \quad \text{ and } \quad S_{2k+2N_{2k}+2} = \til{S}_{2k+2}.
\]
From this construction, it is immediate that $S$ is a simple random walk on~$\Z^d$. Let $\til{\RR}$ be the range of $\til{S}$. From the construction of $S$ from $\til{S}$ we immediately get that 
\begin{align}\label{eq:tilderr}
\til{\RR}_{n}= \RR_{n+2N_{n}} = \RR_{n+2N_{n-1}},
\end{align}
where the second equality follows, since adding the double backtracks does not change the range.  

\begin{lemma}
	Let $\til{S}$ be a random walk on $\Z^d$ starting from $0$ with no double backtracks at even times. If $\til{R}$ stands for its range, then for any positive constants $c$ and $c'$ we have
	\[
	\pr{\sum_{x\in \til{\RR}_{2n}}\sum_{y\in \til{\RR}[2n,(2+c')n]} G(x,y) \geq c\sqrt{n}} \to 0\text{ as } n\to\infty.
	\]
\end{lemma}

\begin{proof}[\bf Proof]

Let $M$ be the number of double backtracks added during the interval $[2n,(2+c')n]$, i.e.,
\begin{align}\label{eq:evenmsum}
M=\sum_{\substack{i=2n\\ i\text{ even}}}^{(2+c')n}\xi_i.
\end{align}
Then, we have that 
\begin{align*}
	\til{\RR}[2n,(2+c')n] \subseteq \RR[2n+2N_{2(n-1)}, (2+c')n+ 2N_{2(n-1)} + 2M].
\end{align*} 
 Note that the inclusion above could be strict, 
since $\til{S}$ does not allow double backtracks, while $S$ does so.
We now can write 
\begin{align*}
&\pr{\sum_{x\in \til{\RR}_{2n}}\sum_{y\in \til{\RR}[2n,(2+c')n]} G(x,y) \geq c\sqrt{n}} 
\\&\leq \pr{\sum_{x\in \RR[0,2n+2N_{2(n-1)}]}\sum_{y\in \RR[2n+2N_{2(n-1)},(2+c')n+2N_{2(n-1)}+2M]} G(x,y) \geq c\sqrt{n}}	\\
&\leq  \pr{\sum_{x\in \RR[0,2n+2N_{2(n-1)}]}\sum_{y\in \RR[2n+2N_{2(n-1)},(2+2C+c')n+2N_{2(n-1)}]} G(x,y) \geq c\sqrt{n}} + \pr{M\geq Cn}.
\end{align*}
By~\eqref{eq:evenmsum} and Chebyshev's inequality 
we obtain that for some positive $C$, $\pr{M\geq Cn}$ vanishes
as~$n$ tends to infinity. Since $G(x-a,y-a) = G(x,y)$ 
for all~$x,y,a$, it follows that 
\newpage
\begin{align*}
	\pr{\sum_{x\in \RR[0,2n+2N_{2(n-1)}]}\sum_{y\in \RR[2n+N_{2(n-1)},(2+2C+c')n+2N_{2(n-1)}]} G(x,y) \geq c\sqrt{n}}\\ = \pr{\sum_{x\in \RR_1}\sum_{y\in \RR_2} G(x,y) \geq c\sqrt{n}},
\end{align*}
	where $\RR_1 = \RR[0,2n+2N_{2(n-1)}] - S_{2n+2N_{2(n-1)}}$ and $\RR_2 = \RR[2n+2N_{2(n-1)},(2+2C+c')n+2N_{2(n-1)}]-S_{2n+2N_{2(n-1)}}$. The importance of considering $\RR_1$ up to time $2n+2N_{2(n-1)}$ and not up to time $2n+2N_{2n}$ is in order to make~$\RR_1$ and $\RR_2$ independent. Indeed, this follows since after time~$2n+2N_{2(n-1)}$ the walk $S$ behaves as a simple random walk in $\Z^d$ independent of the past. Hence we can replace $\RR_2$ by $\RR'_{(2+2C+c')n}$, where $\RR'$ is the range of a simple random walk independent of~$\RR_1$. Therefore we obtain
	\begin{align*}
		\pr{\sum_{x\in \RR_1}\sum_{y\in \RR'_{(2+2C+c')n}} G(x,y) \geq c\sqrt{n}} \leq \pr{\sum_{x\in \RR_{(2C'+2)n}}\sum_{y\in \RR'_{(2+2C+c')n}} G(x,y) \geq c\sqrt{n}} \\+ \pr{N_{2(n-1)}\geq C'n}.
	\end{align*}
	    As before, by Chebyshev's inequality for $C'$ large enough $\pr{N_{2(n-1)}\geq  C'n} \to 0$ as $n\to \infty$ and by Markov's inequality and Lemma~\ref{lem:rearrange}
\begin{align*}
\pr{\sum_{x\in \RR_{(2C'+2)n}}\sum_{y\in \RR'_{(2+2C+c')n}} G(x,y) \geq c\sqrt{n}} &\leq \frac{\E{\sum_{x\in \RR_{C'n}}\sum_{y\in \RR'_{(2+2C+c')n}} G(x,y)}}{c\sqrt{n}}\\ &\lesssim \frac{\log n}{\sqrt{n}},
\end{align*}
and this concludes the proof.
\end{proof}

\begin{claim}\label{cl:sllntil}
Let $\til{\RR}$ be the range of $\til{S}$. Then, almost surely
\[
\frac{\cc{\til{\RR}[2k,2k+n]}}{n} \to \alpha_d \cdot \left(\frac{p}{1-p} \right)\quad \text{ as } \quad n\to\infty.
\]	
\end{claim}

\begin{proof}[\bf Proof]
	
	As mentioned already in the Introduction, Jain and Orey~\cite{JainOrey} proved that 
\begin{align}\label{eq:lawoflargenumbers}
\lim_{n\to\infty}
\frac{\cc{\RR_n}}{n} = \alpha_d=\inf_{m}\frac{\E{\cc{\RR_m}}}{m}.
\end{align}
with the limit $\alpha_d$ being strictly positive for $d\geq 5$.

Clearly the range of $\til{S}$ in $[2k,2k+n]$ satisfies
\begin{align*}
\RR[2k+2N_{2k-1}, 2k+2N_{2k-1}+2N'_n]\setminus \{S_{2k+2N_{2k-1}+1}, S_{2k+2N_{2k-1}+2}\}\subseteq \til{\RR}[2k,2k+n]\\	\til{\RR}[2k,2k+n]\subseteq \RR[2k+2N_{2k-1}, 2k+2N_{2k-1}+2N'_n],
\end{align*}	
where $N'_n$ is the number of double backtracks added between times $2k$ and $2k+n$. We now note that after time $2k+2N_{2k-1}$ the walk $S$ behaves as a simple random walk in $\Z^d$. Hence using~\eqref{eq:lawoflargenumbers} and the fact that $N'_n/n\to p/(2(1-p))$ as $n\to \infty$ almost surely it follows that almost surely 
	\begin{align*}
	\lim_{n\to\infty}	\frac{\cc{\RR[2k+2N_{2k-1}, 2k+2N_{2k-1}+2N'_n]}}{n}= \alpha_d \cdot \left(\frac{p}{1-p} \right).
	\end{align*}
	and this concludes the proof.
\end{proof}

\begin{proof}[\bf Proof of Lemma~\ref{lem:variance}]

Let $\til{S}$ be a random walk with no double backtracks at even times and $S$ a simple random walk constructed from $\til{S}$ as described at the beginning of Section~\ref{sec:gammapos}. We thus have $\til{\RR}_n = \RR_{n+2N_n}$ for all $n$. Let $k_n = [(1-p)n]$, $i_n=[(1-p)(n+A\sqrt{n})]$ and $\ell_n = [(1-p)(n-A\sqrt{n})]$ for a constant $A$ to be determined later.  
Then, by Claim~\ref{cl:sllntil} for all $n$ sufficiently large so that $k_n$ and $\ell_n$ are even numbers we have 
\begin{align}\label{eq:78}
	\pr{\cc{\til{\RR}[k_n,i_n]}\geq  \frac{3}{4}\cdot \left(\frac{A\cdot \alpha_d \cdot p}{1-p}\right)\cdot  \sqrt{n}}&\geq\frac{7}{8} \qquad \text{and} \\
	\label{eq:18}	\pr{\sum_{x\in \til{\RR}[0,k_n]}\sum_{y\in \til{\RR}[k_n,i_n]}G(x,y) \leq  \frac{1}{8}\cdot \left(\frac{A\cdot \alpha_d \cdot p}{1-p}\right)\cdot  \sqrt{n}}&\geq \frac{7}{8}
	\end{align}
	and 
	\begin{align}
	\label{eq:118}
		\pr{\cc{\til{\RR}[\ell_n,k_n]}\geq  \frac{3}{4}\cdot \left(\frac{A\cdot \alpha_d \cdot p}{1-p}\right)\cdot  \sqrt{n}}&\geq\frac{7}{8} \qquad \text{and} \\
	\label{eq:128}	\pr{\sum_{x\in \til{\RR}[0,\ell_n]}\sum_{y\in \til{\RR}[\ell_n,k_n]}G(x,y) \leq  \frac{1}{8}\cdot \left(\frac{A\cdot \alpha_d \cdot p}{1-p}\right)\cdot  \sqrt{n}}&\geq \frac{7}{8}
	\end{align}
	We now define the events
	\begin{align*}
		B_n = \left\{ \frac{2N_{\ell_n} - 2\E{N_{\ell_n}}}{\sqrt{n}} \,\,\in \,\,[A+1,A+2]\right\} \quad \text{and} \quad D_n = \left\{ \frac{2N_{i_n} - 2\E{N_{i_n}}}{\sqrt{n}} \,\,\in \,\,[1-A,2-A]\right\}.
	\end{align*}
	Then, for all $n$ sufficiently large we have for a constant $c_A>0$ that depends on $A$
	\begin{align}\label{eq:bndn}
		\pr{B_n} \geq c_A \quad \text{and}\quad \pr{D_n}\geq c_A.
	\end{align}

Since we have already showed the existence of the limit $\vr{\cc{\RR_n}}/n$ as $n$ tends to infinity, it suffices to prove that the limit is strictly positive along a subsequence. So we are only going to take~$n$ such that~$k_n$ is even. Take $n$ sufficiently large so that~\eqref{eq:78} holds and $k_n$ is even. We then consider two cases: 
\begin{align*}
	\rm{(i)}\,\, \pr{\cc{\til{\RR}[0,k_n]}\geq \E{\cc{\RR_n}}}\geq \frac{1}{2} \quad \text{or} \quad \rm{(ii)} \,\, \pr{\cc{\til{\RR}[0,k_n]}\leq \E{\cc{\RR_n}}}\geq \frac{1}{2}.
\end{align*}
We start with case (i). Using Proposition~\ref{prop:capdec} we have 
\begin{align*}
	\cc{\til{\RR}[0,i_n]} \geq \cc{\til{\RR}[0,k_n]} + \cc{\til{\RR}[k_n,i_n]} - 2\sum_{x\in \til{\RR}[0,k_n]}\sum_{y\in \til{\RR}[k_n,i_n]} G(x,y).
\end{align*}
From this, we deduce that
\begin{align}\label{eq:allcc}
\begin{split}
	\pr{\cc{\til{\RR}[0,i_n]}\geq \E{\cc{\RR_n}} + \frac{1}{2}\cdot \left(\frac{A\cdot \alpha_d \cdot p}{1-p} \right) \sqrt{n}} \\ 
	\geq \pr{\cc{\til{\RR}[0,k_n]}\geq \E{\cc{\RR[0,n}} , \cc{\til{\RR}[k_n,i_n]} \geq \frac{3}{4}\cdot \left(\frac{A\cdot \alpha_d \cdot p}{1-p}\right)\cdot  \sqrt{n}} \\-
	\pr{\sum_{x\in \til{\RR}[0,k_n]}\sum_{y\in \til{\RR}[k_n,i_n]}G(x,y) >  \frac{1}{8}\cdot \left(\frac{A\cdot \alpha_d \cdot p}{1-p}\right)\cdot  \sqrt{n}}.
\end{split}
\end{align}
The assumption of case (i) and~\eqref{eq:78} give that 
\begin{align*}
	\pr{\cc{\til{\RR}[0,k_n]}\geq \E{\cc{\RR[0,n}} , \cc{\til{\RR}[k_n,i_n]} \geq \frac{3}{4}\cdot \left(\frac{A\cdot \alpha_d \cdot p}{1-p}\right)\cdot  \sqrt{n}} \geq \frac{3}{8}.
\end{align*}
Plugging this lower bound together with~\eqref{eq:18} into~\eqref{eq:allcc} yields
\begin{align*}
	\pr{\cc{\til{\RR}[0,i_n]}\geq \E{\cc{\RR_n}} + \frac{1}{2}\cdot \left(\frac{A\cdot \alpha_d \cdot p}{1-p} \right) \cdot \sqrt{n}} \geq \frac{1}{4}.
\end{align*}
Since $N$ is independent of $\til{S}$, using~\eqref{eq:bndn} it follows that 
\begin{align*}
	\pr{\cc{\til{\RR}[0,i_n]}\geq \E{\cc{\RR_n}} + \frac{1}{2}\cdot \left(\frac{A\cdot \alpha_d \cdot p}{1-p} \right)\cdot  \sqrt{n}, D_n}\geq \frac{c_A}{4}.
\end{align*}
It is not hard to see that on the event $D_n$ we have $i_n+2N_{i_n}\in [n,n+3\sqrt{n}]$. Therefore, since $\til{\RR}[0,k] = \RR[0,k+2N_k]$ we deduce
\begin{align*}
	\pr{\exists\,\, m\leq 3\sqrt{n}: \, \cc{\RR[0,n+m]} \geq \E{\cc{\RR_n}} + \frac{1}{2}\cdot \left(\frac{A\cdot \alpha_d \cdot p}{1-p} \right)\cdot  \sqrt{n} }\geq \frac{c_A}{4}.
\end{align*} 
Since $\cc{\RR[0,\ell]}$ is increasing in $\ell$, we obtain
\begin{align*}
	\pr{\cc{\RR[0,n+3\sqrt{n}]} \geq \E{\cc{\RR_n}} + \frac{1}{2}\cdot \left(\frac{A\cdot \alpha_d \cdot p}{1-p} \right)\cdot  \sqrt{n} }\geq \frac{c_A}{4}.
\end{align*}
	Using now the deterministic bound $\cc{
	\RR[0,n+3\sqrt{n}]} \leq \cc{\RR[0,n]} + 3\sqrt{n}$ gives
	\begin{align*}
		\pr{\cc{\RR[0,n]} \geq \E{\cc{\RR_n}} + \left(\frac{1}{2}\cdot \left(\frac{A\cdot \alpha_d \cdot p}{1-p} \right)  - 3\right)\cdot \sqrt{n}}\geq \frac{c_A}{4},
	\end{align*}
	and hence choosing $A$ sufficiently large so that 
	\[
	\frac{1}{2}\cdot \left(\frac{A\cdot \alpha_d \cdot p}{1-p} \right)  - 3>0
	\]
and using Chebyshev's inequality shows in case (i) for a strictly positive constant $c$ we have
\[
\vr{\cc{\RR_n}} \geq c\cdot  n.
\]
We now treat case (ii). We are only going to consider $n$ so that $\ell_n$ is even. 
Using Proposition~\ref{prop:capdec} again we have 
\[
\cc{\til{\RR}[0,\ell_n]}\leq \cc{\til{\RR}[0,k_n]} - \cc{\til{\RR}[\ell_n,k_n]} + 2\sum_{x\in \til{\RR}[0,\ell_n]} \sum_{y\in \til{\RR}[\ell_n,k_n]} G(x,y).
\]
Then, similarly as before using~\eqref{eq:118}, \eqref{eq:128} and~\eqref{eq:bndn} we obtain
\begin{align*}
	\pr{\cc{\til{\RR}[0,\ell_n]}\leq \E{\cc{\RR_n}} - \frac{1}{2}\cdot \left(\frac{A\cdot \alpha_d \cdot p}{1-p} \right)\cdot \sqrt{n},\, B_n}\geq \frac{c_A}{4}.
\end{align*}
Since on $B_n$ we have $\ell_n + 2N_{\ell_n}\in [n,n+3\sqrt{n}]$, it follows that
\begin{align*}
	\pr{\exists \,\, m\leq 3\sqrt{n}: \, \cc{\RR[0,n+m]} \leq \E{\cc{\RR_n}} - \frac{1}{2}\cdot \left(\frac{A\cdot \alpha_d \cdot p}{1-p} \right) \cdot \sqrt{n}} \geq \frac{c_A}{4}.
\end{align*}
Using the monotonicity property of $\cc{\RR_l}$ in $\ell$ we finally conclude that 
\[
\pr{\cc{\RR[0,n]} \leq \E{\cc{\RR_n}} - \frac{1}{2}\cdot \left(\frac{A\cdot \alpha_d \cdot p}{1-p} \right) \cdot \sqrt{n}} \geq \frac{c_A}{4},
\]
and hence Chebyshev's inequality again finishes the proof in case (ii).
\end{proof}

\section{Central limit theorem}\label{sec-four}
We start this section by recalling the Lindeberg-Feller theorem.
Then, we give the proof of Theorem~\ref{thm:clt}.

\begin{theorem}[Lindeberg-Feller]\label{thm:lind}
For each $n$ let $(X_{n,i}: \, 1\leq i\leq n)$ be a collection of independent random variables with zero mean. Suppose that the following two conditions are satisfied
\newline
{\rm{(i)}} $\sum_{i=1}^{n}\E{X_{n,i}^2} \to \sigma^2>0$ as $n\to \infty$ and 
\newline
{\rm{(ii)}} $\sum_{i=1}^{n}\E{(X_{n,i})^2\1(|X_{n,i}|>\epsilon)} \to 0$ as $n\to \infty$ for all $\epsilon>0$.
\newline
Then, $S_n=X_{n,1}+\ldots + X_{n,n} \Longrightarrow \sigma \NN(0,1)$ as $n\to \infty$.
\end{theorem}
For a proof we refer the reader to~\cite[Theorem~3.4.5]{Durrett}.

Before proving Theorem~\ref{thm:clt}, we upper bound the fourth moment of~$\overline{\cc{\RR_n}}$. Recall that for a random variable $X$ we write $\overline{X}= X-\E{X}$.

\begin{lemma}\label{lem:fourth}
For all $d\geq 6$ and for all $n$ we have 
\[
\E{(\overline{\cc{\RR_n}})^4}\lesssim n^2.
\]
\end{lemma}

\begin{proof}[\bf Proof]
This proof is similar to the proof of Lemma~\ref{lem:exis}. We only emphasize the points where they differ. Again we write $X_n=\cc{\RR_n}$ and we set for all $k\geq 1$
\[
a_k=\sup\left\{ \left(\E{\overline{X}_n^4}\right)^{1/4}: \, 2^k\leq n< 2^{k+1}\right\}.
\]
For $k\geq 2$ take $n$ such that $2^k\leq n< 2^{k+1}$ and write $n_1=[n/2]$ and $n_2=n-\ell$. Then, Corollary~\ref{cor:decomposition} and the triangle inequality give
\begin{align*}
\|\overline{X}_n\|_4\leq \| \overline{X}_{n_1}+ \overline{X}_{n_2}\|_4 + 4\|\EE(n_1,n_2)\|_4 \leq \left( \E{\overline{X}_{n_1}^4} + \E{\overline{X}_{n_2}^4} + 6\E{\overline{X}_{n_1}^2} \E{\overline{X}_{n_2}^2}  \right)^{1/4}+c_1\log n,
\end{align*}
where the last inequality follows from Lemma~\ref{lem:powers} and the fact that $\overline{X}_{n_1}$ and $\overline{X}_{n_2}$ are independent. Using Lemma~\ref{lem:variance} we get that 
\[
\E{\overline{X}_{n_1}^2} \E{\overline{X}_{n_2}^2} \asymp n^2.
\]
Also using the obvious inequality for $a,b>0$ that $(a+b)^{1/4}\leq a^{1/4}+ b^{1/4}$ we obtain
\begin{align*}
\|\overline{X}_n\|_4\leq \left( \E{\overline{X}_{n_1}^4} + \E{\overline{X}_{n_2}^4} \right)^{1/4} + c_2\sqrt{n}.
\end{align*}
We deduce that 
\begin{align*}
a_k\leq 2^{1/4} a_{k-1} +c_3 2^{k/2}.
\end{align*}
Setting $b_k= 2^{-k/2} a_k$ we get
\[
b_{k}\leq \frac{1}{2^{1/4}}b_{k-1} + c_3,
\]
This implies that $(b_k,k\in \N)$ 
is a bounded sequence, and hence $a_k\leq C 2^{k/2}$ for a positive constant $C$, or in other words,
\[
\left(\E{\overline{X}_n^4}\right)^{1/4} \lesssim \sqrt{n}
\]
and this concludes the proof.
\end{proof}

\begin{proof}[\bf Proof of Theorem~\ref{thm:clt}]

From Corollary~\ref{cor:decomposition} we have 
\begin{align}\label{eq:bigeq}
\sum_{i=1}^{2^L} \cc{\RR^{(i)}_{n/2^L}} - 2\sum_{\ell=1}^{L}\sum_{i=1}^{2^{\ell-1}} \EE_{\ell}^{(i)}\leq \cc{\RR_n} \leq \sum_{i=1}^{2^L} \cc{\RR^{(i)}_{n/2^L}},
\end{align}
where $(\cc{\RR^{(i)}_{n/2^L}})$ are independent for different $i$'s and $\RR^{(i)}_{n/2^L}$  has the same law as $\RR_{[n/2^L]}$ or~$\RR_{[n/2^L+1]}$ and for each $\ell$ the random variables $(\EE_{\ell}^{(i)})$ are independent and have the same law as $\sum_{x\in \RR^{(i)}_{n/2^L}} \sum_{y\in \til{\RR}^{(i)}_{n/2^L}} G(x,y)$, with $\til{\RR}$ and independent copy of $\RR$.

To simplify notation we set $X_{i,L} =  \cc{\RR^{(i)}_{n/2^L}}$ and $X_n=\cc{\RR_n}$ and for convenience we rewrite~\eqref{eq:bigeq} as
\begin{align}\label{eq:keyeq}
\sum_{i=1}^{2^L} X_{i,L} - 2\sum_{\ell=1}^{L}\sum_{i=1}^{2^{\ell-1}} \EE_{\ell}^{(i)}\leq X_n \leq \sum_{i=1}^{2^L} X_{i,L}.
\end{align}
 We now let 
\begin{align*}
\EE(n) = \sum_{i=1}^{2^L}\overline{X}_{i,L} - \overline{X}_n.
\end{align*}
Using inequality~\eqref{eq:keyeq} we get 
\begin{align*}
\E{|\EE(n)|} \leq 4\E{\sum_{\ell=1}^{L}\sum_{i=1}^{2^{\ell-1}} \EE_{\ell}^{(i)}}\lesssim\sum_{\ell=1}^{L} 2^\ell \log n \lesssim 2^L \log n,
\end{align*}
where the penultimate inequality follows from Lemma~\ref{lem:powers} for $k=1$ and the fact that $f_d(n)\leq \log n$ for all $d\geq 6$.

Choosing $L$ so that $2^L= n^{1/4}$ gives $\E{|\EE(n)|}/\sqrt{n}\to 0$ as $n\to \infty$.  We can thus reduce the problem of showing that $\overline{X}_n/{\sqrt{n}}$ converges in distribution to showing that $\sum_{i=1}^{2^L}\overline{X}_{i,L}/\sqrt{n}$ converges to a normal random variable. 

We now focus on proving that 
\begin{align}\label{eq:goal}
\frac{\sum_{i=1}^{2^L} \overline{X}_{i,L}}{\sqrt{n}} \Longrightarrow \sigma \NN(0,1) \quad \text{as } n\to \infty.
\end{align}
We do so by invoking Lindeberg-Feller's Theorem~\ref{thm:lind}. From Lemma~\ref{lem:variance} we immediately get that as $n$ tends to infinity.
\[
\sum_{i=1}^{2^L} \frac{1}{n}\cdot\vr{\overline{X}_{i,L}} \sim \frac{2^L}{n} \cdot \gamma_d\cdot \frac{n}{2^L} = \gamma_d>0, 
\]
which means that the first condition of Lindeberg-Feller is satisfied. It remains to check the second one, i.e.,
\begin{align*}
\lim_{n\to\infty}
\sum_{i=1}^{2^L} \frac{1}{n}\cdot \E{\overline{X}_{i,L}^2\1(|\overline{X}_{i,L}|>\epsilon \sqrt{n})} = 0.
\end{align*}
By Cauchy-Schwartz, we have
\begin{align*}
\E{\overline{X}_{i,L}^2\1(|\overline{X}_{i,L}|>\epsilon \sqrt{n})}\leq \sqrt{\E{(\overline{X}_{i,L})^4}\pr{|\overline{X}_{i,L}|>\epsilon \sqrt{n}}}.
\end{align*}
By Chebyshev's inequality and using that $\vr{\overline{X}_{i,L}}\sim \gamma_d \cdot n/2^L$ from Lemma~\ref{lem:variance} we get
\begin{align*}
\pr{|\overline{X}_{i,L}|>\epsilon \sqrt{n}} \leq \frac{1}{\epsilon^2 2^L}.
\end{align*}
Using Lemma~\ref{lem:fourth} we now get
\begin{align*}
\sum_{i=1}^{2^L} \frac{1}{n}\cdot \E{\overline{X}_{i,L}^2\1(|\overline{X}_{i,L}|>\epsilon \sqrt{n})} \lesssim \sum_{i=1}^{2^L} \frac{1}{n} \cdot \frac{n}{2^L} \frac{1}{\epsilon 2^{L/2}} = \frac{1}{\epsilon 2^{L/2}} \to 0,
\end{align*}
since $L=\log n/4$. Therefore, the second condition of Lindeberg-Feller Theorem~\ref{thm:lind} is satisfied and this finishes the proof.
\end{proof}

\section{Rough estimates in $d=4$ and $d=3$}
\label{sec-five}
\begin{proof}[\bf Proof of Corollary~\ref{lem:d4}]
In order to use Lawler's Theorem~\ref{thm:lawler}, we 
introduce a random walk $\widetilde S$ starting at the origin
and independent from $S$,
with distribution denoted $\widetilde {\mathbb{P}}$. Then, as noticed already by Jain and Orey~\cite[Section~2]{JainOrey}, the capacity
of the range reads (with the convention $\RR_{-1}=\emptyset$)
\be\label{new-1}
\cc{\RR_n}=\sum_{k=0}^n \1(S_k\notin\RR_{k-1})\times
\widetilde {\mathbb{P}}_{S_k}\big((S_k+\widetilde \RR_\infty)
\cap \RR_n=\emptyset\big),
\ee
where~$\til{\RR}_\infty= \til{\RR}[1,\infty)$.

Thus, for $k$ fixed we can consider three independent walks.
The first is $S^1:[0,k]\to\Z^d$ with $S^1_i=S_k-S_{k-i}$, 
the second is $S^2:[0,n-k]\to\Z^d$ with $S^2_i:=S_{k+i}-S_k$, and
the third $S^3\equiv \widetilde S$. With these symbols, 
equality \eqref{new-1} reads
$$
\cc{\RR_n}=\sum_{k=0}^n \1 (0\notin \RR^1[1,k] )\times 
\widetilde {\mathbb{P}}\big(\RR^3[1,\infty)\cap(\RR^1[0,k]\cup
\RR^2[0,n-k])=\emptyset\big).
$$
Then, taking expectation with respect to $S^1$, $S^2$
and $S^3$, we get
\be\label{new-3}
\E{\cc{\RR_n}}=\sum_{k=0}^n 
\pr{0\not\in \RR^1[1,k],\ \RR^3[1,\infty)\cap
(\RR^1[0,k]\cup\RR^2[0,n-k])=\emptyset}.
\ee
Now, $\varepsilon \in(0,1/2)$ being fixed, we define $\varepsilon_n:=\varepsilon n/\log n$, and 
 divide the above sum into two subsets: when $k$ is smaller than $\varepsilon_n$ or larger than $n-\varepsilon_n$, and when $k$ is in between. The terms in the first subset can be bounded just by one, 
and we obtain this way the following upper bound.
$$
\E{\cc{\RR_n}}\ \le\ 2\varepsilon_n 
 +  n \pr{0\not\in \RR^1[1,\varepsilon_n],\ \RR^3[1,\varepsilon_n]\cap
(\RR^1[0,\varepsilon_n]\cup\RR^2[0,\varepsilon_n])=\emptyset}.
$$
Since this holds for any $\varepsilon>0$, and $\log \varepsilon_n \sim \log n$, we conclude using \eqref{lawler-key}, that
\be\label{new-5}
\limsup_{n\to\infty}\ \frac{\log n}{n} \times \E{\cc{\RR_n}}\ \le\ 
\frac{\pi^2}{8}.
\ee
For the lower bound, we first observe that \eqref{new-3} gives 
$$\E{\cc{\RR_n}}\ \ge \ n \,  
\pr{0\not\in \RR^1[1,n],\ \RR^3[1,\infty]\cap
(\RR^1[0,n]\cup\RR^2[0,n])=\emptyset},
$$
and we conclude the proof using \eqref{lawler-key}. 
\end{proof}

\begin{proof}[\bf Proof of Proposition~\ref{lem:d3}]

We recall $L_n(x)$ is the local time at $x$, i.e.,
\[
L_n(x) = \sum_{i=0}^{n-1}\1(S_i=x).
\]
The lower bound is obtained using the representation \eqref{variation-capa}, as we choose $\nu(x)=L_n(x)/n$. This gives 
\be\label{new-6}
\cc{\RR_n}\ge \frac{n}{\frac{1}{n}\sum_{x,y\in \Z^d}
G(x,y)L_n(x)L_n(y)},
\ee
and using Jensen's inequality, we deduce
\be\label{new-7}
\E{\cc{\RR_n}}\ge \frac{n}{\frac{1}{n}\sum_{x,y\in \Z^d}
G(x,y)\E{L_n(x)L_n(y)}}.
\ee
Note that
\be\label{new-8}
\sum_{x,y\in \Z^d}G(x,y)\E{L_n(x)L_n(y)}=
\sum_{0\le k\leq n} \sum_{0\le k'\leq n} \E{G(S_k,S_{k'})}.
\ee
We now obtain, using the local CLT,
\begin{align*}
	\sum_{0\le k\leq n} \sum_{0\le k'\leq n} \E{G(S_k,S_{k'})} = \sum_{0\le k\leq n} \sum_{0\le k'\leq n} \E{G(0,S_{|k'-k|})}\lesssim \sum_{0\le k\leq n} \sum_{0\le k'\leq n} \E{\frac{1}{1+|S_{|k'-k|}|}}  \lesssim n\sqrt{n}
\end{align*}
and this gives the desired lower bound. 
For the upper bound one can use that in dimension $3$,
$$\cc{A} \ \lesssim \ \textrm{rad}(A),$$
where $\textrm{rad}(A)=\sup_{x\in A} \| x \|$ (see \cite[Proposition~2.2.1(a) and (2.16)]{Lawlerinter}). 
Therefore Doob's inequality gives 
$$\E{\cc{\RR_n}} \ \lesssim\ \E{\sup_{k\le n}\ \|S_k\|}\ \lesssim \ \sqrt n$$
and this completes the proof.
\end{proof}

\section{Open Questions}\label{sec-six}

We focus on open questions 
concerning the typical behaviour 
of the capacity of the range.

Our main inequality \reff{main-lower} is reminiscent of
the equality for the range
\be\label{main-range}
|\RR[0,2n]|=|\RR[0,n]|+|\RR[n,2n]|-|\RR[0,n]\cap \RR[n,2n]|.
\ee
However, the {\it intersection term} $|\RR[0,n]\cap \RR[n,2n]|$ has
a different asymptotics for $d\ge 3$
\be\label{conj-0}
\E{|\RR[0,n]\cap \RR[n,2n]|}\asymp f_{d+2}(n).
\ee
This leads us to {\it add two dimensions} when comparing
the volume of the range with respect to the capacity of the range. 
It is striking that the volume of the range in $d=1$ is typically
of order $\sqrt n$ as the capacity
of the range in $d=3$. The fact that the volume
of the range in $d=2$ is typically of order
$n/\log n$ like the capacity of the range in $d=4$ is as striking.
Thus, based on these analogies,
we conjecture that the variance in dimension five behaves as follows.
\be\label{conj-2}
\vr{\cc{\RR_n}}\asymp n\log n.
\ee
Note that an upper bound with a similar nature
as \reff{main-lower} is lacking, and that \reff{key-lawler}
is of a different order of magnitude. Indeed,
\[
\E{\cc{\RR[0,n]\cap \RR[n,2n]}}\, \le\,  \E{|\RR[0,n]\cap \RR[n,2n]|}\, \lesssim\, 
 f_{d+2}(n).
\]
Another question would be to show a concentration result in dimension 4, i.e., 
\be\label{conj-1}
\frac{\cc{\RR_n}}{\E{\cc{\RR_n}}}\quad \stackrel{\text{(P)}}{\longrightarrow}
\quad 1.
\ee
We do not expect \reff{conj-1} to hold in dimension three, but rather
that the limit would be random.

\section*{Acknowledgements} 

We thank the Institute IM\'eRA in Marseille for its hospitality. This work has been carried out thanks partially to the support of
A$^*$MIDEX grant (ANR-11-IDEX-0001-02) funded by the French Government
``Investissements d'Avenir" program.

\bibliographystyle{plain}
\bibliography{biblio}
\end{document}